\documentclass[12pt]{amsart}
\usepackage{amsmath,amsthm,amsfonts,amssymb,mathrsfs}
\date{\today}

\usepackage[cp1251]{inputenc}         % Кодова сторінка Windows1251

\usepackage[ukrainian]{babel} % Підтримка кирилиці
\usepackage{amsmath,amsthm,amsfonts,amssymb,mathrsfs}
\usepackage{eucal}

\usepackage{color}

\usepackage{amssymb,cite}
\usepackage[colorlinks,plainpages,citecolor=magenta, linkcolor=blue, backref]{hyperref}%,bookmarksnumbered,

\usepackage{hyperref}

\newtheorem{theorem}{Теорема}%[section]

\newtheorem{proposition}{Твердження}
\newtheorem{corollary}{Наслiдок}
\newtheorem{lemma}{Лема}
\theoremstyle{definition}
\newtheorem{example}{Приклад}%[section]
\newtheorem{remark}{Зауваження}%[section]

\newtheorem{definition}{Означення}%[section]

\begin{document}

\title[Про одне узагальнення бiциклiчного моно\"{\i}да]{Про одне узагальнення бiциклiчного моно\"{\i}да}

\author[Олег~Гутік, Микола Михаленич]{Олег~Гутік, Микола Михаленич}
\address{Механіко-математичний факультет, Львівський національний університет ім. Івана Франка, Університецька 1, Львів, 79000, Україна}
\email{oleg.gutik@lnu.edu.ua,
ovgutik@yahoo.com, myhalenychmc@gmail.com}

\keywords{semigroup, bicyclic monoid, extension}
\subjclass[2020]{20M15,  20M50, 18B40.}

\begin{abstract}
Вводимо алгебричне розширення $\boldsymbol{B}_{\omega}^{\mathscr{F}}$ біциклічного моноїда для довільної $\omega$-замк\-не\-ної сім'ї $\mathscr{F}$ підмножин в $\omega$, які узагальнюють біциклічний моноїд, зліченну напівгрупу матричних одиниць і деякі інші комбінаторні інверсні напівгрупи. Доведено, що $\boldsymbol{B}_{\omega}^{\mathscr{F}}$ є комбінаторною інверсною напівгрупою, а також описано відношення Ґріна, частковий природний порядок на напівгрупі $\boldsymbol{B}_{\omega}^{\mathscr{F}}$ та її множину ідемпотентів. Також доведено критерії  простоти, $0$-простоти, біпростоти та $0$-біпростоти напівгрупи $\boldsymbol{B}_{\omega}^{\mathscr{F}}$, а також коли $\boldsymbol{B}_{\omega}^{\mathscr{F}}$ містить одиницю, ізоморфна біциклічному моноїду або зліченній напівгрупі матричних одиниць.

\emph{\textbf{Ключові слова:}} Напівгрупа, біциклічний моноїд, розширення.

\bigskip
\noindent
\emph{Oleg Gutik, Mykola Mykhalenych, \textbf{On some generalization of the bicyclic monoid},} Visnyk Lviv. Univ. Ser. Mech. Math. \textbf{90} (2020), 5--19.

\smallskip
\noindent
We introduce an algebraic extension $\boldsymbol{B}_{\omega}^{\mathscr{F}}$ of the bicyclic monoid  for an arbitrary $\omega$-closed family $\mathscr{F}$ subsets of $\omega$  which generalizes the bicyclic monoid, the countable semigroup of matrix units and some other combinatorial inverse semigroups. It is proved that $\boldsymbol{B}_{\omega}^{\mathscr{F}}$ is a combinatorial inverse semigroup and Green's relations, the natural partial order on $\boldsymbol{B}_{\omega}^{\mathscr{F}}$, and its set of idempotents are described. We provide criteria of simplicity, $0$-simplicity, bisimplicity, $0$-bisimplicity of the semigroup $\boldsymbol{B}_{\omega}^{\mathscr{F}}$ and when $\boldsymbol{B}_{\omega}^{\mathscr{F}}$ has the identity, is isomorphic to the bicyclic semigroup or the countable semigroup of matrix units.
\end{abstract}

\maketitle

%\section{Термінологія та означення}

\section{\textbf{Вступ}}\label{section-1}

Ми користуватимемось термінологією з \cite{Clifford-Preston-1961, Clifford-Preston-1967, Lawson-1998, Petrich1984}.
Надалі у тексті множину невід'ємних цілих чисел  позначатимемо через $\omega$. Для довільного $k\in\omega$ позначимо $[k)=\{i\in\omega\colon i\geqslant k\}$.

Підмножина $A$ в $\omega$ називається \emph{індуктивною}, якщо з того, що $i\in A$ випливає, що $i+1\in A$. Очевидно, що $\varnothing$ --- індуктивна множина в $\omega$, і непорожня підмножина $A\subseteq\omega$ є індук\-тивною тоді і лише тоді, коли $A=[k)$ для деякого $k\in\omega$.

Якщо $S$~--- напівгрупа, то її підмножина ідемпотентів позначається через $E(S)$.  На\-пів\-гру\-па $S$ називається \emph{інверсною}, якщо для довільного її елемента $x$ існує єдиний елемент $x^{-1}\in S$ такий, що $xx^{-1}x=x$ та $x^{-1}xx^{-1}=x^{-1}$ \cite{Vagner-1952, Petrich1984}. В інверсній напівгрупі $S$ вище означений елемент $x^{-1}$ називається \emph{інверсним до} $x$. \emph{В'язка}~--- це напівгрупа ідемпотентів, а \emph{напівґратка}~--- це комутативна в'язка.

Якщо $S$ --- напівгрупа, то ми позначатимемо відношення Ґріна на $S$ через $\mathscr{R}$, $\mathscr{L}$, $\mathscr{D}$, $\mathscr{H}$ і $\mathscr{J}$ (див. означення в \cite[\S2.1]{Clifford-Preston-1961} або \cite{Green-1951}). Напівгрупа $S$ називається \emph{простою}, якщо $S$ не містить власних двобічних ідеалів, тобто $S$ складається з одного $\mathscr{J}$-класу, і \emph{біпростою}, якщо $S$ складається з одного $\mathscr{D}$-класу.

Відношення еквівалентності $\mathfrak{K}$ на напівгрупі $S$ називається \emph{конгруенцією}, якщо для елементів $a$ та $b$ напівгрупи $S$ з того, що виконується умова $(a,b)\in\mathfrak{K}$ випливає, що $(ca,cb), (ad,bd) \in\mathfrak{K}$, для довільних $c,d\in S$. Відношення $(a,b)\in\mathfrak{K}$ також будемо записувати $a\mathfrak{K}b$, і в цьому випадку будемо говорити, що \emph{елементи $a$ i $b$  $\mathfrak{K}$-еквівалентні}.

Якщо $S$~--- напівгрупа, то на $E(S)$ визначено частковий порядок:
$
e\preccurlyeq f
$   тоді і лише тоді, коли
$ef=fe=e$.
Так означений частковий порядок на $E(S)$ називається \emph{при\-род\-ним}.

Означимо відношення $\preccurlyeq$ на інверсній напівгрупі $S$ так:
$
    s\preccurlyeq t
$
тоді і лише тоді, коли $s=te$.
для деякого ідемпотента $e\in S$. Так означений частковий порядок назива\-єть\-ся \emph{при\-род\-ним част\-ковим порядком} на інверсній напівгрупі $S$~\cite{Vagner-1952}. Очевидно, що звуження природного часткового порядку $\preccurlyeq$ на інверсній напівгрупі $S$ на її в'язку $E(S)$ є при\-род\-ним частковим порядком на $E(S)$.

Нагадаємо (див.  \cite[\S1.12]{Clifford-Preston-1961}, що \emph{біциклічною напівгрупою} (або \emph{біциклічним моноїдом}) ${\mathscr{C}}(p,q)$ називається напівгрупа з одиницею, породжена двоелементною мно\-жи\-ною $\{p,q\}$ і визначена одним  співвідношенням $pq=1$. Біциклічна на\-пів\-група відіграє важливу роль у теорії
на\-півгруп. Зокрема, класична теорема О.~Ан\-дерсена \cite{Andersen-1952}  стверд\-жує, що {($0$-)}прос\-та напівгрупа з (ненульовим) ідем\-по\-тен\-том є цілком {($0$-)}прос\-тою тоді і лише тоді, коли вона не містить ізоморфну копію бі\-циклічного моноїда. Різні розширення біциклічного моноїда вводили раніше різ\-ні автори \cite{Fortunatov-1976, Fotedar-1974, Fotedar-1978, Warne-1967}. Такими є, зокрема, конструкції Брука та Брука--Рейлі занурення напівгруп у прості та описання інверсних біпростих і $0$-біпростих $\omega$-напівгруп \cite{Bruck-1958, Reilly-1966, Warne-1966, Gutik-2018}.

\begin{remark}\label{remark-10}
Легко бачити, що біциклічний моноїд ${\mathscr{C}}(p,q)$ ізоморфний напівгрупі, заданій на множині $\boldsymbol{B}_{\omega}=\omega\times\omega$ з напівгруповою операцією
\begin{align*}
  (i_1,j_1)\cdot(i_2,j_2)&=(i_1+i_2-\min\{j_1,i_2\},j_1+j_2-\min\{j_1,i_2\})=\\
  &=
\left\{
  \begin{array}{ll}
    (i_1-j_1+i_2,j_2), & \hbox{якщо~} j_1\leqslant i_2;\\
    (i_1,j_1-i_2+j_2), & \hbox{якщо~} j_1\geqslant i_2.
  \end{array}
\right.
\end{align*}
\end{remark}

Ми вводимо алгебричні розширення $\boldsymbol{B}_{\omega}^{\mathscr{F}}$ біциклічного моноїда для довільної $\omega$-замк\-не\-ної сім'ї $\mathscr{F}$ підмножин в $\omega$, які узагальнюють біциклічний моноїд, зліченну напівгрупу матричних одиниць і деякі інші комбінаторні інверсні напівгрупи. Доведено, що $\boldsymbol{B}_{\omega}^{\mathscr{F}}$ є комбінаторною інверсною напівгрупою, а також описано відношення Ґріна, частковий природний порядок на напівгрупі $\boldsymbol{B}_{\omega}^{\mathscr{F}}$ та її множину ідемпотентів. Також доведено критерії  простоти, $0$-простоти, біпростоти та $0$-біпростоти напівгрупи $\boldsymbol{B}_{\omega}^{\mathscr{F}}$, а також коли $\boldsymbol{B}_{\omega}^{\mathscr{F}}$ містить одиницю, ізоморфна біциклічному моноїду або зліченній напівгрупі матричних одиниць.

%%%%%%%%%%%%%%%%%%%%%%%%%%%%%%%%%%%%%%%%%%%%%%%%%%%%%%%%%%%%%%%%%%%%%%%%%%%%%%%%%%
\section{\textbf{Конструкція розширення $\boldsymbol{B}_{\omega}^{\mathscr{F}}$}}\label{section-2}

Нехай $\mathscr{P}(\omega)$~--- сім'я усіх підмножин ординалу $\omega$.
Для довільних $F\in\mathscr{P}(\omega)$ i $n,m\in\omega$ приймемо
\begin{equation*}
n-m+F=\{n-m+k\colon k\in F\}, \qquad \hbox{якщо} \; F\neq\varnothing
\end{equation*}
i $n-m+F=\varnothing$, якщо $F=\varnothing$.

Будемо говорити, що підсім'я  $\mathscr{F}\subseteq\mathscr{P}(\omega)$ є \emph{${\omega}$-замкненою}, якщо $F_1\cap(-n+F_2)\in\mathscr{F}$ для довільних $n\in\omega$ i $F_1,F_2\in\mathscr{F}$.

Нехай $\boldsymbol{B}_{\omega}$~--- біциклічний моноїд і  $\mathscr{F}$ --- ${\omega}$-замкнена підсім'я в  $\mathscr{P}(\omega)$. На множині $\boldsymbol{B}_{\omega}\times\mathscr{F}$ означимо бінарну операцію ``$\cdot$''  формулою
\begin{equation*}
  (i_1,j_1,F_1)\cdot(i_2,j_2,F_2)=
  \left\{
    \begin{array}{ll}
      (i_1-j_1+i_2,j_2,(j_1-i_2+F_1)\cap F_2), & \hbox{якщо~} j_1\leqslant i_2;\\
      (i_1,j_1-i_2+j_2,F_1\cap (i_2-j_1+F_2)), & \hbox{якщо~} j_1\geqslant i_2.
    \end{array}
  \right.
\end{equation*}

Тоді отримуємо, що
\begin{align*}
  ((i_1,j_1,F_1)&\cdot(i_2,j_2,F_2)) \cdot (i_3,j_3,F_3)=\\
  &=
  \left\{
    \begin{array}{ll}
      (i_1{-}j_1{+}i_2,j_2,(j_1{-}i_2{+}F_1)\cap F_2){\cdot}(i_3,j_3,F_3), & \hbox{якщо~} j_1\leqslant i_2;\\
      (i_1,j_1{-}i_2{+}j_2,F_1\cap (i_2{-}j_1{+}F_2)){\cdot}(i_3,j_3,F_3), & \hbox{якщо~} j_1\geqslant i_2
    \end{array}
  \right.
{=}\\
    &{=}
    \left\{
    \begin{array}{l}
      (i_1{-}j_1{+}i_2{-}j_2{+}i_3,j_3,(j_2{-}i_3{+}((j_1{-}i_2{+}F_1)\cap F_2))\cap F_3),\\
         \qquad \qquad\qquad\qquad\qquad\qquad\qquad
         \hbox{якщо~} j_1 \leqslant i_2 \hbox{~i~} j_2 \leqslant i_3;\\
      (i_1{-}j_1{+}i_2,j_2{-}i_3{+}j_3,(j_1{-}i_2{+}F_1)\cap F_2\cap(i_3{-}j_2{+}F_3)),\\
         \qquad \qquad\qquad\qquad\qquad\qquad\qquad
         \hbox{якщо~} j_1 \leqslant i_2 \hbox{~i~} j_2 \geqslant i_3;\\
      (i_1{-}j_1{+}i_2{-}j_2{+}i_3,j_3,j_1{-}i_2{+}j_2{-}i_3{+}(F_1\cap (i_2{-}j_1{+}F_2))\cap F_3),\\
         \qquad \qquad\qquad\qquad\qquad\qquad\qquad
         \hbox{якщо~} j_1 \geqslant i_2 \hbox{~i~} j_1{-}i_2{+}j_2 \leqslant i_3;\\
      (i_1,j_1{-}i_2{+}j_2{-}i_3{+}j_3,(F_1\cap (i_2{-}j_1{+}F_2))\cap(j_3{-}j_1{+}i_2{-}j_2{+}F_3)),\\
         \qquad \qquad\qquad\qquad\qquad\qquad\qquad
         \hbox{якщо~} j_1 \geqslant i_2 \hbox{~i~} j_1{-}i_2{+}j_2 \geqslant i_3
    \end{array}
  \right.
=\\
    &=
    \left\{
    \begin{array}{llr}
      (i_1{-}j_1{+}i_2{-}j_2{+}i_3,j_3,(j_2{-}i_3{+}j_1{-}i_2{+}F_1)\cap(j_2{-}i_3{+} F_2)\cap F_3),\\
         \qquad \qquad\qquad\qquad\qquad\qquad\qquad
         \hbox{якщо~} j_1 \leqslant i_2 \hbox{~i~} j_2 \leqslant i_3;                             &&(1_1)\\
      (i_1{-}j_1{+}i_2,j_2{-}i_3{+}j_3,(j_1{-}i_2{+}F_1)\cap F_2\cap(i_3{-}j_2{+}F_3)),\\
         \qquad \qquad\qquad\qquad\qquad\qquad\qquad
         \hbox{якщо~} j_1 \leqslant i_2 \hbox{~i~} j_2 \geqslant i_3;                             &&(2_1)\\
      (i_1{-}j_1{+}i_2{-}j_2{+}i_3,j_3,(j_1{-}i_2{+}j_2{-}i_3{+}F_1)\cap (j_2{-}i_3{+}F_2)\cap F_3),\\
         \qquad \qquad\qquad\qquad\qquad\qquad\qquad
         \hbox{якщо~} j_1 \geqslant i_2 \hbox{~i~} j_1{-}i_2{+}j_2 \leqslant i_3;                 &&(3_1)\\
      (i_1,j_1{-}i_2{+}j_2{-}i_3{+}j_3,F_1\cap (i_2{-}j_1{+}F_2)\cap(j_3{-}j_1{+}i_2{-}j_2{+}F_3)),\\
         \qquad \qquad\qquad\qquad\qquad\qquad\qquad
         \hbox{якщо~} j_1 \geqslant i_2 \hbox{~i~} j_1{-}i_2{+}j_2 \geqslant i_3                  &&(4_1)
    \end{array}
  \right.
\end{align*}
i
\begin{align*}
(i_1,j_1,F_1)& \cdot((i_2,j_2,F_2)\cdot(i_3,j_3,F_3))=\\
&=\left\{
    \begin{array}{ll}
      (i_1,j_1,F_1){\cdot}(i_2{-}j_2{+}i_3,j_3,(j_2{-}i_3{+}F_2)\cap F_3), & \hbox{якщо~} j_2\leqslant i_3;\\
      (i_1,j_1,F_1){\cdot}(i_2,j_2{-}i_3{+}j_3,F_2\cap (i_3{-}j_2{+}F_3)), & \hbox{якщо~} j_2\geqslant i_3
    \end{array}
  \right.
=
\\
  & =
\left\{
    \begin{array}{l}
      (i_1{-}j_1{+}i_2{-}j_2{+}i_3,j_3,(j_1{-}i_2{+}j_2{-}i_3{+}F_1)\cap(j_2{-}i_3{+}F_2)\cap F_3),\\
         \qquad \qquad\qquad\qquad\qquad\qquad\qquad \hbox{якщо~} j_2 \leqslant i_3 \hbox{~i~} j_1 \leqslant i_2{-}j_2{+}i_3;\\
      (i_1,j_1{-}i_2{+}j_2{-}i_3{+}j_3,F_1\cap(i_2{-}j_2{+}i_3{-}j_1{+}((j_2{-}i_3{+}F_2)\cap F_3))),\\
         \qquad \qquad\qquad\qquad\qquad\qquad\qquad \hbox{якщо~} j_2 \leqslant i_3 \hbox{~i~} j_1 \geqslant i_2{-}j_2{+}i_3;\\
      (i_1{-}j_1{+}i_2,j_2{-}i_3{+}j_3,(j_1{-}j_2{+}i_3{-}j_3{+}F_1)\cap F_2\cap (i_3{-}j_2{+}F_3)),\\
         \qquad \qquad\qquad\qquad\qquad\qquad\qquad \hbox{якщо~} j_2 \geqslant i_3 \hbox{~i~} j_1 \leqslant i_2;\\
      (i_1,j_1{-}i_2{+}j_2{-}i_3{+}j_3,F_1\cap(i_2{-}j_1{+}(F_2\cap (i_3{-}j_2{+}F_3)))),\\
         \qquad \qquad\qquad\qquad\qquad\qquad\qquad \hbox{якщо~} j_2 \geqslant i_3 \hbox{~i~} j_1 \geqslant i_2;
    \end{array}
  \right.
=
\end{align*}

\begin{align*}
 \qquad \qquad &=
\left\{
    \begin{array}{llr}
      (i_1{-}j_1{+}i_2{-}j_2{+}i_3,j_3,(j_1{-}i_2{+}j_2{-}i_3{+}F_1)\cap(j_2{-}i_3{+}F_2)\cap F_3),\\
         \qquad \qquad\qquad\qquad\qquad\qquad\qquad \hbox{якщо~} j_2 \leqslant i_3 \hbox{~i~} j_1 \leqslant i_2{-}j_2{+}i_3; && (1_2)\\
      (i_1,j_1{-}i_2{+}j_2{-}i_3{+}j_3,F_1\cap(i_2{-}j_1{+}F_2)\cap(i_2{-}j_2{+}i_3{-}j_1{+}F_3)),\\
         \qquad \qquad\qquad\qquad\qquad\qquad\qquad \hbox{якщо~} j_2 \leqslant i_3 \hbox{~i~} j_1 \geqslant i_2{-}j_2{+}i_3; && (2_2)\\
      (i_1{-}j_1{+}i_2,j_2{-}i_3{+}j_3,(j_1{-}j_2{+}i_3{-}j_3{+}F_1)\cap F_2\cap (i_3{-}j_2{+}F_3)),\\
         \qquad \qquad\qquad\qquad\qquad\qquad\qquad \hbox{якщо~} j_2 \geqslant i_3 \hbox{~i~} j_1 \leqslant i_2;             && (3_2)\\
      (i_1,j_1{-}i_2{+}j_2{-}i_3{+}j_3,F_1\cap(i_2{-}j_1{+}F_2)\cap (i_2{-}j_1{+}i_3{-}j_2{+}F_3)),\\
         \qquad \qquad\qquad\qquad\qquad\qquad\qquad \hbox{якщо~} j_2\geqslant i_3 \hbox{~i~} j_1\geqslant i_2.               && (4_2)
    \end{array}
  \right.
\end{align*}

Аналогічно, як і у випадку доведення асоціативності напівгрупової операції для біциклічного моноїда, маємо такі еквівалентності умов:
\begin{equation*}
  (1_1)\vee(3_1)\Longleftrightarrow (1_2), \qquad (2_1)\Longleftrightarrow(3_2), \qquad (4_1)\Longleftrightarrow(2_2)\vee(4_2).
\end{equation*}

Отже, ми довели таке твердження.

\begin{proposition}\label{proposition-2.1}
Якщо сім'я  $\mathscr{F}\subseteq\mathscr{P}(\omega)$ є ${\omega}$-замкненою, то $(\boldsymbol{B}_{\omega}\times\mathscr{F},\cdot)$ є напівгрупою.
\end{proposition}

Припустимо, що ${\omega}$-замкнена сім'я $\mathscr{F}\subseteq\mathscr{P}(\omega)$ містить порожню множину $\varnothing$, то з означення напівгрупової операції $(\boldsymbol{B}_{\omega}\times\mathscr{F},\cdot)$ випливає, що множина
$ %\begin{equation*}
  \boldsymbol{I}=\{(i,j,\varnothing)\colon i,j\in\omega\}
$ %\end{equation*}
є ідеалом напівгрупи $(\boldsymbol{B}_{\omega}\times\mathscr{F},\cdot)$.

\begin{definition}\label{definition-2.2}
Для довільної ${\omega}$-замкненої сім'ї $\mathscr{F}\subseteq\mathscr{P}(\omega)$ означимо
\begin{equation*}
  \boldsymbol{B}_{\omega}^{\mathscr{F}}=
\left\{
  \begin{array}{ll}
    (\boldsymbol{B}_{\omega}\times\mathscr{F},\cdot)/\boldsymbol{I}, & \hbox{якщо~} \varnothing\in\mathscr{F};\\
    (\boldsymbol{B}_{\omega}\times\mathscr{F},\cdot), & \hbox{якщо~} \varnothing\notin\mathscr{F}.
  \end{array}
\right.
\end{equation*}
\end{definition}

%%%%%%%%%%%%%%%%%%%%%%%%%%%%%%%%%%%%%%%%%%%%%%%%%%%%%%%%%%%%%%%%%%%%%%%%%%%%%%%%%%
\section{\textbf{Алгебричні властивості напівгрупи $\boldsymbol{B}_{\omega}^{\mathscr{F}}$}}\label{section-3}

Надалі ми вважатимемо, що $\mathscr{F}$ --- ${\omega}$-замкнена підсім'я в  $\mathscr{P}(\omega)$.

Твердження леми \ref{lemma-3.1} очевидне.

\begin{lemma}\label{lemma-3.1}
Напівгрупа $\boldsymbol{B}_{\omega}^{\mathscr{F}}$ містить нуль  $\boldsymbol{0}$ тоді і лише тоді, коли $\varnothing\in\mathscr{F}$, причому нуль $\boldsymbol{0}$ є образом ідеала $\boldsymbol{I}$ при природному гомоморфізмі, породженому конгруенцією Ріса
\begin{equation*}
\boldsymbol{\mathfrak{C}}_{\boldsymbol{I}}=\left\{(x,x)\colon x\in\boldsymbol{B}_{\omega}\times\mathscr{F} \right\}\cup(\boldsymbol{I}\times\boldsymbol{I})
\end{equation*}
на напівгрупі $(\boldsymbol{B}_{\omega}\times\mathscr{F},\cdot)$.
\end{lemma}

\begin{lemma}\label{lemma-3.2}
Ненульовий елемент $(i,j,F)$ напівгрупи $\boldsymbol{B}_{\omega}^{\mathscr{F}}$ є ідемпотентом тоді і лише тоді, коли $i=j$.
\end{lemma}

\begin{proof}
$(\Longleftarrow)$ Якщо $i=j$, то маємо, шо
\begin{equation*}
(i,i,F)\cdot(i,i,F)=(i,i,F\cap F)=(i,i,F),
\end{equation*}
а  отже, $(i,i,F)\in E(\boldsymbol{B}_{\omega}^{\mathscr{F}})$.

$(\Longrightarrow)$ Припустимо, що $(i,j,F)\in E(\boldsymbol{B}_{\omega}^{\mathscr{F}})$ та $i\geqslant j$. Тоді
\begin{equation*}
  (i,j,F)\cdot(i,j,F)=(i-j+i,j,(j-i+F)\cap F)=(i,j,F),
\end{equation*}
а отже, $i-j+i=i$, звідки випливає, що $i=j$. У випадку $i\leqslant j$ аналогічно отримуємо, що $i=j$.
\end{proof}

\begin{lemma}\label{lemma-3.3}
Ідемпотенти напівгрупи $\boldsymbol{B}_{\omega}^{\mathscr{F}}$ комутують.
\end{lemma}

\begin{proof}
Очевидно, якщо напівгрупа $\boldsymbol{B}_{\omega}^{\mathscr{F}}$ містить нуль $\boldsymbol{0}$, то елемент $\boldsymbol{0}$ комутує з кожним елементом напівгрупи $\boldsymbol{B}_{\omega}^{\mathscr{F}}$.

Зафіксуємо два довільні ненульові ідемпотенти напівгрупи $\boldsymbol{B}_{\omega}^{\mathscr{F}},$ які за лемами \ref{lemma-3.1} і \ref{lemma-3.2} мають зображення $(i,i,F_i)$ i $(j,j,F_j)$, де $i,j\in\omega$ i $F_1, F_2\in \mathscr{F}$ --- непорожні підмножини в $\omega$. Тоді
\begin{equation*}
\begin{split}
  (i,i,F_i)\cdot(j,j,F_j)&=\left\{
    \begin{array}{ll}
      (i-i+j,j,(i-j+F_i)\cap F_j), & \hbox{якщо~} i\leqslant j;\\
      (i,i-j+j,F_i\cap (j-i+F_j)), & \hbox{якщо~} i\geqslant j
    \end{array}
  \right.= \\
    & =\left\{
    \begin{array}{ll}
      (j,j,(i-j+F_i)\cap F_j), & \hbox{якщо~} i\leqslant j;\\
      (i,i,F_i\cap (j-i+F_j)), & \hbox{якщо~} i\geqslant j
    \end{array}
  \right.
\end{split}
\end{equation*}
i
\begin{equation*}
\begin{split}
  (j,j,F_i)\cdot(i,i,F_j)&=\left\{
    \begin{array}{ll}
      (j,j-i+i,F_j\cap (i-j+F_i)), & \hbox{якщо~} i\leqslant j;\\
      (i-i+j,j,(j-i+F_j)\cap F_i), & \hbox{якщо~} i\geqslant j
    \end{array}
  \right.= \\
    & =\left\{
    \begin{array}{ll}
      (j,j,(i-j+F_i)\cap F_j), & \hbox{якщо~} i\leqslant j;\\
      (i,i,F_i\cap (j-i+F_j)), & \hbox{якщо~} i\geqslant j,
    \end{array}
  \right.
\end{split}
\end{equation*}
звідки
\begin{equation*}
(i,i,F_i)\cdot(j,j,F_j)=(j,j,F_i)\cdot(i,i,F_j),
\end{equation*}
а отже, справджується твердження леми.
\end{proof}

\begin{lemma}\label{lemma-3.4}
Елементи $(i,j,F)$ і $(j,i,F)$ є інверсними в напівгрупі $\boldsymbol{B}_{\omega}^{\mathscr{F}}$.
\end{lemma}

\begin{proof}
Справді, маємо
\begin{equation*}
(i,j,F)\cdot(j,i,F)\cdot(i,j,F)=(i,i,F)\cdot(i,j,F)=(i,j,F)
\end{equation*}
і
\begin{equation*}
(j,i,F)\cdot(i,j,F)\cdot(j,i,F)=(j,j,F)\cdot(j,i,F)=(j,i,F),
\end{equation*}
звідки й випливає твердження леми.
\end{proof}

За теоремою Вагнера--Престона (див. \cite[теорема~1.17]{Clifford-Preston-1961}) регулярна напівгрупа є інверсною тоді і лише тоді, коли ідемпотенти в ній комутують, а отже, з лем \ref{lemma-3.3} і \ref{lemma-3.4} випливає теорема \ref{theorem-3.5}.

\begin{theorem}\label{theorem-3.5}
Якщо $\mathscr{F}$ --- ${\omega}$-замкнена підсім'я в  $\mathscr{P}(\omega)$, то $\boldsymbol{B}_{\omega}^{\mathscr{F}}$ --- інверсна напівгрупа.
\end{theorem}

Лема \ref{lemma-3.6} описує природний частковий порядок на множині ідемпотентів напівгрупи $\boldsymbol{B}_{\omega}^{\mathscr{F}}$.

\begin{lemma}\label{lemma-3.6}
$(i,i,F_i)\preccurlyeq(j,j,F_j)$ в $E(\boldsymbol{B}_{\omega}^{\mathscr{F}})$ тоді і тільки тоді, коли
\begin{equation*}
i\geqslant j \qquad \hbox{i} \qquad F_i\subseteq j-i+F_j.
\end{equation*}
\end{lemma}

\begin{proof}
Припустимо, що $(i,i,F_i)\preccurlyeq(j,j,F_j)$ в $E(\boldsymbol{B}_{\omega}^{\mathscr{F}})$. Тоді
\begin{equation*}
(i,i,F_i)\cdot(j,j,F_j)=(i,i,F_i),
\end{equation*}
і оскільки   \begin{align*}
  (i,i,F_i)\cdot(j,j,F_j)&
=
  \left\{
    \begin{array}{ll}
      (i-i+j,j,(i-j+F_i)\cap F_j), & \hbox{якщо~} i\leqslant j;\\
      (i,i-j+j,F_i\cap (j-i+F_j)), & \hbox{якщо~} i\geqslant j
    \end{array}
  \right.=\\
&
=
  \left\{
    \begin{array}{ll}
      (j,j,(i-j+F_i)\cap F_j), & \hbox{якщо~} i\leqslant j;\\
      (i,i,F_i\cap (j-i+F_j)), & \hbox{якщо~} i\geqslant j,
    \end{array}
  \right.
\end{align*}
то з леми \ref{lemma-3.3} випливає, що $i\geqslant j$ i $F_i\subseteq j-i+F_j$.

Навпаки, припустимо, що $i\geqslant j$ i $F_i\subseteq j-i+F_j$. Тоді
\begin{align*}
  (i,i,F_i)\cdot(j,j,F_j)&=(i,i-j+j,F_i\cap (j-i+F_j))=\\
  &=(i,i,F_i\cap (j-i+F_j))=\\
  &=(i,i,F_i),
\end{align*}
а отже, $(i,i,F_i)\preccurlyeq(j,j,F_j)$ в $E(\boldsymbol{B}_{\omega}^{\mathscr{F}})$.
\end{proof}

Твердження \ref{proposition-3.8} описує природний частковий порядок на напівгрупі $\boldsymbol{B}_{\omega}^{\mathscr{F}}$.

\begin{proposition}\label{proposition-3.8}
Нехай $(i_1,j_1,F_1)$ i $(i_2,j_2,F_2)$~--- ненульові елементи напівгрупи $\boldsymbol{B}_{\omega}^{\mathscr{F}}$. Тоді $(i_1,j_1,F_1)\preccurlyeq(i_2,j_2,F_2)$ тоді і лише тоді, коли $F_1\subseteq -k+F_2$ й $i_1-i_2=j_1-j_2=k$ для деякого $k\in\omega$.
\end{proposition}

\begin{proof}
$(\Longleftarrow)$ Якщо $F_1\subseteq- k+F_2$, $i_1-i_2=k$ i $j_1-j_2=k$ для деякого $k\in\omega$, то
\begin{align*}
  (i_1,j_1,F_1)\cdot(i_1,j_1,F_1)^{-1}\cdot(i_2,j_2,F_2)&=(i_1,j_1,F_1)\cdot(j_1,i_1,F_1)\cdot(i_2,j_2,F_2)= \\
   &=(i_1,i_1,F_1)\cdot(i_2,j_2,F_2)=\\
   &=(i_1,i_1-i_2+j_2,F_1\cap (i_2-i_1+F_2))=\\
   &=(i_1,j_1-j_2+j_2,F_1\cap (-k+F_2))=\\
   &=(i_1,j_1,F_1\cap (-k+F_2))=\\
   &=(i_1,j_1,F_1),
\end{align*}
і за лемою~1.4.6 \cite{Lawson-1998} отримуємо, що $(i_1,j_1,F_1)\preccurlyeq(i_2,j_2,F_2)$.

\smallskip
$(\Longrightarrow)$ Припустимо, що $(i_1,j_1,F_1)\preccurlyeq(i_2,j_2,F_2)$. Тоді за лемою~1.4.6 \cite{Lawson-1998},
\begin{align*}
  (i_1,j_1,F_1)&=(i_1,j_1,F_1)\cdot(i_1,j_1,F_1)^{-1}\cdot(i_2,j_2,F_2)=\\
   &=(i_1,j_1,F_1)\cdot(j_1,i_1,F_1)\cdot(i_2,j_2,F_2)= \\
   &=(i_1,i_1,F_1)\cdot(i_2,j_2,F_2)=\\
   &=\left\{
    \begin{array}{ll}
      (i_1-i_1+i_2,j_2,(i_1-i_2+F_1)\cap F_2), & \hbox{якщо~} i_1\leqslant i_2;\\
      (i_1,i_1-i_2+j_2,F_1\cap (i_2-i_1+F_2)), & \hbox{якщо~} i_1\geqslant i_2
    \end{array}
  \right.  =\\
   &=\left\{
    \begin{array}{llllr}
      (i_2,j_2,(i_1-i_2+F_1)\cap F_2), & \hbox{якщо~} i_1\leqslant i_2;        &&&(1_3)\\
      (i_1,i_1-i_2+j_2,F_1\cap (i_2-i_1+F_2)), & \hbox{якщо~} i_1\geqslant i_2 &&&(2_3)
    \end{array}
  \right.
\end{align*}
i
\begin{align*}
  (i_1,j_1,F_1)&=(i_2,j_2,F_2)\cdot(i_1,j_1,F_1)^{-1}\cdot(i_1,j_1,F_1)=\\
   &=(i_2,j_2,F_2)\cdot(j_1,i_1,F_1)\cdot(i_1,j_1,F_1)= \\
   &=(i_2,j_2,F_2)\cdot(j_1,j_1,F_1)=\\
   &=\left\{
    \begin{array}{ll}
      (i_2,j_2-j_1+j_1,F_2\cap (j_1-j_2+F_1)), & \hbox{якщо~} j_1\leqslant j_2;\\
      (i_2-j_2+j_1,j_1,(j_2-j_1+F_2)\cap F_1), & \hbox{якщо~} j_1\geqslant j_2
    \end{array}
  \right.  =\\
   &=\left\{
    \begin{array}{llllr}
      (i_2,j_2,F_2\cap (j_1-j_2+F_1)),         & \hbox{якщо~} j_1\leqslant j_2;        &&&(1_4)\\
      (i_2-j_2+j_1,j_1,(j_2-j_1+F_2)\cap F_1), & \hbox{якщо~} j_1\geqslant j_2         &&&(2_4)
    \end{array}
  \right.
\end{align*}
У випадках $(1_3)$ i $(1_4)$ маємо, що $i_1=i_2$, $j_1=j_2$, а отже, $i_1-i_2=j_1-j_2=0$ i $F_1\subseteq -0+F_2$. У випадках $(2_3)$ i $(2_4)$ одержуємо, що $i_1=i_2-j_2+j_1$, і врахувавши, що $i_1\geqslant i_2$, $j_1\geqslant j_2$ і $(j_2-j_1+F_2)\cap F_1=F_1$, та прийнявши $k=i_1-i_2=j_1-j_2$, отримуємо, що $F_1\subseteq -k+F_2$ для деякого $k\in\omega$.
\end{proof}

Теорема~\ref{theorem-3.7} описує відношення Ґріна на напівгрупі $\boldsymbol{B}_{\omega}^{\mathscr{F}}$.

\begin{theorem}\label{theorem-3.7}
Нехай $\mathscr{F}$ --- ${\omega}$-замкнена підсім'я в  $\mathscr{P}(\omega)$. Тоді:
\begin{itemize}
  \item[$\boldsymbol{(i)}$] $(i_1,j_1,F_1)\mathscr{R}(i_2,j_2,F_2)$ в $\boldsymbol{B}_{\omega}^{\mathscr{F}}$ тоді і лише тоді, коли $i_1=i_2$ i $F_1=F_2$;
  \item[$\boldsymbol{(ii)}$] $(i_1,j_1,F_1)\mathscr{L}(i_2,j_2,F_2)$ в $\boldsymbol{B}_{\omega}^{\mathscr{F}}$ тоді і лише тоді, коли $j_1=j_2$ i $F_1=F_2$;
  \item[$\boldsymbol{(iii)}$] $(i_1,j_1,F_1)\mathscr{H}(i_2,j_2,F_2)$ в $\boldsymbol{B}_{\omega}^{\mathscr{F}}$ тоді і лише тоді, коли $i_1=i_2$, $j_1=j_2$ i $F_1=F_2$, а отже, всі $\mathscr{H}$-класи напівгрупи $\boldsymbol{B}_{\omega}^{\mathscr{F}}$ є одноелементними;
  \item[$\boldsymbol{(iv)}$] $(i_1,j_1,F_1)\mathscr{D}(i_2,j_2,F_2)$ в $\boldsymbol{B}_{\omega}^{\mathscr{F}}$ тоді і лише тоді, коли  $F_1=F_2$;
  \item[$\boldsymbol{(v)}$] $(i_1,j_1,F_1)\mathscr{J}(i_2,j_2,F_2)$ в $\boldsymbol{B}_{\omega}^{\mathscr{F}}$ тоді і лише тоді, коли  існують $k_1,k_2\in\omega$ такі, що $F_1\subseteq -k_1+F_2$ і $F_2\subseteq -k_2+F_1$.
\end{itemize}
\end{theorem}

\begin{proof}
$\boldsymbol{(i)}$  Нехай $(i_1,j_1,F_1)$ i $(i_2,j_2,F_2)$~--- $\mathscr{R}$-еквівалентні елементи напівгрупи $\boldsymbol{B}_{\omega}^{\mathscr{F}}$ такі, що
$(i_1,j_1,F_1)\mathscr{R}(i_2,j_2,F_2)$. Позаяк за теоремою \ref{theorem-3.5}, $\boldsymbol{B}_{\omega}^{\mathscr{F}}$~--- інверсна напівгрупа та $$(i_1,j_1,F_1)\boldsymbol{B}_{\omega}^{\mathscr{F}}=(i_2,j_2,F_2)\boldsymbol{B}_{\omega}^{\mathscr{F}},$$ то з теореми~1.17
\cite{Clifford-Preston-1961} ви\-пли\-ває, що
\begin{equation*}
  (i_1,j_1,F_1)\boldsymbol{B}_{\omega}^{\mathscr{F}}=(i_1,j_1,F_1)(i_1,j_1,F_1)^{-1}\boldsymbol{B}_{\omega}^{\mathscr{F}}= (i_1,i_1,F_1)\boldsymbol{B}_{\omega}^{\mathscr{F}}
\end{equation*}
і
\begin{equation*}
  (i_2,j_2,F_2)\boldsymbol{B}_{\omega}^{\mathscr{F}}=(i_2,j_2,F_2)(i_2,j_2,F_2)^{-1}\boldsymbol{B}_{\omega}^{\mathscr{F}}= (i_2,i_2,F_2)\boldsymbol{B}_{\omega}^{\mathscr{F}},
\end{equation*}
а отже, $(i_1,i_1,F_1)=(i_2,i_2,F_2)$. Отож виконуються рівності $i_1=i_2$ i $F_1=F_2$.

Навпаки, нехай $(i_1,j_1,F_1)$ i $(i_2,j_2,F_2)$~--- елементи напівгрупи $\boldsymbol{B}_{\omega}^{\mathscr{F}}$ такі, що $i_1=i_2$ i $F_1=F_2$. Тоді $(i_1,i_1,F_1)=(i_2,i_2,F_2)$. Позаяк $\boldsymbol{B}_{\omega}^{\mathscr{F}}$~--- інверсна напівгрупа, то з теореми~1.17
\cite{Clifford-Preston-1961} випливає, що
\begin{align*}
  (i_1,j_1,F_1)\boldsymbol{B}_{\omega}^{\mathscr{F}}&=(i_1,j_1,F_1)(i_1,j_1,F_1)^{-1}\boldsymbol{B}_{\omega}^{\mathscr{F}}=\\
   &=(i_1,i_1,F_1)\boldsymbol{B}_{\omega}^{\mathscr{F}}=\\
   &=(i_2,j_2,F_2)(i_2,j_2,F_2)^{-1}\boldsymbol{B}_{\omega}^{\mathscr{F}}=\\
   &=(i_2,i_2,F_2)\boldsymbol{B}_{\omega}^{\mathscr{F}},
\end{align*}
а отже, $(i_1,j_1,F_1)\mathscr{R}(i_2,j_2,F_2)$ в $\boldsymbol{B}_{\omega}^{\mathscr{F}}$.

\smallskip
Доведення твердження $\boldsymbol{(ii)}$ аналогічне до доведення твердження $\boldsymbol{(i)}$.

\smallskip
Твердження  $\boldsymbol{(iii)}$ випливає з $\boldsymbol{(i)}$ та $\boldsymbol{(ii)}$.

\smallskip
$\boldsymbol{(iv)}$ Нехай $(i_1,j_1,F_1)$ i $(i_2,j_2,F_2)$~--- елементи напівгрупи $\boldsymbol{B}_{\omega}^{\mathscr{F}}$ такі, що \linebreak $(i_1,j_1,F_1)\mathscr{D}(i_2,j_2,F_2)$. Позаяк $\mathscr{R},\mathscr{L}\subseteq\mathscr{D}$  та
\begin{equation*}
(i_1,i_1,F_1)\mathscr{R}(i_1,j_1,F_1)\qquad  \hbox{i} \qquad (i_2,j_2,F_2)\mathscr{L}(j_2,j_2,F_2)
\end{equation*}
за твердженнями $\boldsymbol{(i)}$ і $\boldsymbol{(ii)}$, то $(i_1,i_1,F_1)\mathscr{D}(j_2,j_2,F_2)$. За лемою Манна (див. \cite[лема 1.1]{Munn-1966} або \cite[твердження 3.2.5]{Lawson-1998}) існує елемент $(i,j,F)\in\boldsymbol{B}_{\omega}^{\mathscr{F}}$ такий, що
\begin{equation*}
(i,j,F)\cdot(i,j,F)^{-1}=(i_1,i_1,F_1)\qquad \hbox{i}\qquad (i,j,F)^{-1}\cdot(i,j,F)=(j_2,j_2,F_2).
\end{equation*}
Але за лемою \ref{lemma-3.4} маємо, що
\begin{equation*}
(i,j,F)\cdot(i,j,F)^{-1}=(i,j,F)\cdot(j,i,F)=(i,i,F)
\end{equation*}
i
\begin{equation*}
 (i,j,F)^{-1}\cdot(i,j,F)=(j,i,F)\cdot(i,j,F)=(j,j,F),
\end{equation*}
а отже, $F=F_1=F_2$.

Нехай $i_1,i_2,j_1,j_2$~--- довільні невід'ємні цілі числа та $F\in\mathscr{F}$. За твердженнями $\boldsymbol{(i)}$ і $\boldsymbol{(ii)}$ отримуємо, що $(i_1,i_1,F)\mathscr{R}(i_1,j_1,F)$ i $(i_2,j_2,F)\mathscr{L}(j_2,j_2,F)$, а з леми \ref{lemma-3.2} випливає, що $(i_1,i_1,F),(i_2,j_2,F)\in E(\boldsymbol{B}_{\omega}^{\mathscr{F}})$. Позаяк
\begin{equation*}
(i_1,j_2,F)\cdot(i_1,j_2,F)^{-1}=(i_1,j_2,F)\cdot(j_2,i_1,F)=(i_1,i_1,F)
\end{equation*}
і
\begin{equation*}
(i_1,j_2,F)^{-1}\cdot(i_1,j_2,F)=(j_2,i_1,F)\cdot(i_1,j_2,F)=(j_2,j_2,F),
\end{equation*}
то за лемою Манна (див. \cite[лема 1.1]{Munn-1966} або \cite[твердження 3.2.5]{Lawson-1998}) отримуємо, що $(i_1,i_1,F)\mathscr{D}(j_2,j_2,F)$ в $\boldsymbol{B}_{\omega}^{\mathscr{F}}$, і позаяк $\mathscr{R}\circ\mathscr{D}\circ\mathscr{L}\subseteq \mathscr{D}$, то $(i_1,j_1,F)\mathscr{D}(i_2,j_2,F)$ в $\boldsymbol{B}_{\omega}^{\mathscr{F}}$.

\smallskip
$\boldsymbol{(v)}$ Зауважимо, оскільки $\mathscr{D}\subseteq \mathscr{J}$ і за твердженням $\boldsymbol{(iv)}$ маємо, що  \linebreak $(0,0,F)\mathscr{D}(i,j,F)$ в $\boldsymbol{B}_{\omega}^{\mathscr{F}}$ для довільних $i,j\in\omega$ i $F\in\mathscr{F}$, то достатньо знайти необхідну та достатню умову, коли елементи $(0,0,F_1)$ i $(0,0,F_2)$ є $\mathscr{J}$-еквівалентні в $\boldsymbol{B}_{\omega}^{\mathscr{F}}$.

За твердженням~3.2.8 \cite{Lawson-1998} елементи $a$ i $b$ інверсної напівгрупи $S$ є $\mathscr{J}$-еквівалентні тоді і лише тоді, коли $a\mathscr{D}b'\preccurlyeq b$ i $b\mathscr{D}a'\preccurlyeq a$ для деяких $a',b'\in S$. Тоді за твердженням \ref{proposition-3.8} ідемпотенти $(0,0,F_1)$ i $(0,0,F_2)$ є $\mathscr{J}$-еквівалентні в $\boldsymbol{B}_{\omega}^{\mathscr{F}}$ тоді і лише тоді, коли існують $k_1,k_2\in\omega$ такі, що $F_1\subseteq -k_1+F_2$ і $F_2\subseteq -k_2+F_1$. Тоді з вище сказаного випливає, що $(i_1,j_1,F_1)\mathscr{J}(i_2,j_2,F_2)$ в $\boldsymbol{B}_{\omega}^{\mathscr{F}}$ тоді і лише тоді, коли існують $k_1,k_2\in\omega$ такі, що $F_1\subseteq -k_1+F_2$ і $F_2\subseteq -k_2+F_1$.
\end{proof}

Нагадаємо \cite{Lawson-1998, Ash-1979}, шо інверсна напівгрупа $S$ називається \emph{комбінаторною}, якщо відношення Ґріна $\mathscr{H}$ на $S$ є відношенням рівності.

З твердження $\boldsymbol{(iii)}$ теореми~\ref{theorem-3.7} випливає наслідок \ref{corollary-1111}.

\begin{corollary}\label{corollary-1111}
Якщо $\mathscr{F}$ --- ${\omega}$-замкнена підсім'я в  $\mathscr{P}(\omega)$, то $\boldsymbol{B}_{\omega}^{\mathscr{F}}$ --- комбінаторна інверсна напівгрупа.
\end{corollary}

Також з твердження $\boldsymbol{(v)}$ теореми~\ref{theorem-3.7} випливає наслідок \ref{corollary-1112}.

\begin{corollary}\label{corollary-1112}
Нехай $\mathscr{F}$ --- ${\omega}$-замкнена підсім'я в  $\mathscr{P}(\omega)$ і $\varnothing\notin \mathscr{F}$. Тоді напівгрупа $\boldsymbol{B}_{\omega}^{\mathscr{F}}$ є простою тоді і лише тоді, коли для довільних $F_1,F_2\in \mathscr{F}$ існують $k_1,k_2\in\omega$ такі, що
\begin{equation*}
F_1\subseteq -k_1+F_2 \qquad \hbox{і} \qquad F_2\subseteq -k_2+F_1.
\end{equation*}
\end{corollary}

Нагадаємо \cite{Clifford-Preston-1961}, шо напівгрупа $S$ з нулем $0$ називається \emph{$0$-простою}, якщо $S\cdot S\neq \{0\}$ i $\{0\}$ --- єдиний власний двобічний ідеал в $S$. Добре відомо (див. \cite[лема~2.28]{Clifford-Preston-1961}), що напівгрупа $S$ з нулем $0$ є $0$-простою тоді і лише тоді, коли $S$ має лише два $\mathscr{J}$-класи: $S\setminus\{0\}$ i $\{0\}$. З твердження $\boldsymbol{(v)}$ теореми~\ref{theorem-3.7} випливає наслідок \ref{corollary-1113}.

\begin{corollary}\label{corollary-1113}
Нехай $\mathscr{F}$ --- ${\omega}$-замкнена підсім'я в  $\mathscr{P}(\omega)$ і $\varnothing\in \mathscr{F}$. Тоді напівгрупа $\boldsymbol{B}_{\omega}^{\mathscr{F}}$ є $0$-простою тоді і лише тоді, коли для довільних непорожніх множин $F_1,F_2\in \mathscr{F}$ існують $k_1,k_2\in\omega$ такі, що
\begin{equation*}
F_1\subseteq -k_1+F_2 \qquad \hbox{і} \qquad F_2\subseteq -k_2+F_1.
\end{equation*}
\end{corollary}

Нагадаємо \cite{Saito-1965} (див. також \cite{Lawson-1998}), що інверсна напівгрупа $S$ називається \emph{$E$-унітарною}, якщо для $e\in E(S)$ i $s\in S$ з $e\preccurlyeq s$ випливає, що $s\in E(S)$. Тоді з твердження \ref{proposition-3.8} випливає теорема \ref{theorem-3.9}.

\begin{theorem}\label{theorem-3.9}
Якщо $\mathscr{F}$ --- ${\omega}$-замкнена підсім'я в  $\mathscr{P}(\omega)$ і $\varnothing\notin \mathscr{F}$, то  $\boldsymbol{B}_{\omega}^{\mathscr{F}}$~--- $E$-унітарна інверсна напівгрупа.
\end{theorem}

Індуктивні підмножини в $\omega$ описує лема \ref{ind-sets}.

\begin{lemma}\label{ind-sets}
Непорожня множина $F\subseteq \omega$ є індуктивною в $\omega$ тоді і лише тоді, коли
\begin{equation*}
(-1+F)\cap F=F.
\end{equation*}
\end{lemma}

\begin{proof}
$(\Longrightarrow)$ Припустимо, що $F$ --- непорожня індуктивна множина в $\omega$ і нехай $x\in F$. Тоді $y=x+1\in F$, а отже, $x=y-1\in (-1+F)\cap F$. Звідки випливає, що $F\subseteq (-1+F)\cap F$.

\smallskip
$(\Longleftarrow)$ Нехай $(-1+F)\cap F=F$ для деякої непорожньої множини $F\subseteq \omega$. За\-фік\-су\-ємо довільний  елемент $x\in F$. Тоді існує елемент $y\in F$ такий, що $y-1=x\in F$, а отже, $x+1=y\in F$.
\end{proof}

\begin{theorem}\label{theorem-3.10}
Нехай $\mathscr{F}$ --- ${\omega}$-замкнена підсім'я в  $\mathscr{P}(\omega)$. Напівгрупа $\boldsymbol{B}_{\omega}^{\mathscr{F}}$ містить одиницю тоді і лише тоді, коли сім'я $\mathscr{F}$ містить індуктивну множину $F_0$ в $\omega$ таку, що $F_0=\bigcup\mathscr{F}$. За виконання цих умов елемент $(0,0,F_0)$ є одиницею напівгрупи $\boldsymbol{B}_{\omega}^{\mathscr{F}}$.
\end{theorem}

\begin{proof}
$(\Longrightarrow)$ Припустимо, що напівгрупа $\boldsymbol{B}_{\omega}^{\mathscr{F}}$ містить одиницю. Оскільки одиниця кожної напівгрупи є ідемпотентом, то за лемою \ref{lemma-3.2} одиниця в $\boldsymbol{B}_{\omega}^{\mathscr{F}}$ має вигляд $(i,i,F_0)$ для деяких $i\in\omega$ i $F_0\in \mathscr{F}$. Якщо $\mathscr{F}=\{\varnothing\}$, то твердження теореми оче\-вид\-не, а тому надалі вважатимемо, що $F_0\neq\varnothing$. Якщо $i\neq 0$, то
\begin{equation*}
  (0,0,F_0)\cdot (i,i,F_0)=(i,i,(-i+F_0)\cap F_0)\neq(0,0,F_0),
\end{equation*}
звідки випливає, що елемент $(0,0,F_0)$ є одиницею напівгрупи $\boldsymbol{B}_{\omega}^{\mathscr{F}}$. З рівності
\begin{equation*}
  (0,0,F_0)\cdot(1,1,F_0)=(1,1,(-1+F_0)\cap F_0)=(1,1,F_0)
\end{equation*}
випливає, що $F_0$ --- індуктивна множина в $\omega$. Позаяк
\begin{equation*}
  (0,0,F_0)\cdot(0,0,F)=(0,0,F_0\cap F)=(0,0,F)
\end{equation*}
для довільної множини $F\in \mathscr{F}$, то $F_0\supseteq\bigcup\mathscr{F}$, а з того, що $F_0\in \mathscr{F}$ отримуємо рівність $F_0=\bigcup\mathscr{F}$.

\smallskip
$(\Longleftarrow)$ Доведемо, що елемент $(0,0,F_0)$ є у випадку, коли сім'я $\mathscr{F}$ містить індуктивну множину $F_0$ в $\omega$ таку, що $F_0=\bigcup\mathscr{F}$, є одиницею напівгрупи $\boldsymbol{B}_{\omega}^{\mathscr{F}}$. Нехай $(i,j,F)$~--- довільний елемент напівгрупи $\boldsymbol{B}_{\omega}^{\mathscr{F}}$. Оскільки $F_0$~--- індуктивна підмножина в $\omega$, то $F_0\subseteq ((-i+F_0)\cap\omega)$, і з рівності $F_0=\bigcup\mathscr{F}$ випливає, що $(-i+F_0)\cap F=F$ для довільних $i\in\omega$ i $F\in\mathscr{F}$, а отже, отримуємо, що
\begin{equation*}
  (0,0,F_0)\cdot(i,j,F)=(i,j,(-i+F_0)\cap F)=(i,j,F)
\end{equation*}
i
\begin{equation*}
  (i,j,F)\cdot(0,0,F_0)=(i,j,(-i+F_0)\cap F)=(i,j,F).
\end{equation*}
Теорему доведено.
\end{proof}

\begin{proposition}\label{proposition-3.11}
Напівгрупа $\boldsymbol{B}_{\omega}^{\mathscr{F}}$ ізоморфна біциклічному моноїду ${\mathscr{C}}(p,q)$ тоді і тільки тоді, коли $\mathscr{F}=\{F\}$ i $F$~--- непорожня індуктивна підмножина в $\omega$.
\end{proposition}

\begin{proof}
$(\Longleftarrow)$ Нехай $F$~--- індуктивна підмножина в $\omega$ і $\mathscr{F}=\{F\}$. Для довільних $i,j,k,l\in\omega$ маємо, що
\begin{align*}
  (i,j,F)\cdot (k,l,F)&=
\left\{
  \begin{array}{ll}
    (i-j+k,l,(j-k+F)\cap F), & \hbox{якщо~} j\leqslant k; \\
    (i,j-k+l,F\cap (k-j+F)), & \hbox{якщо~} j\geqslant k
  \end{array}
\right.= \\
  &=
\left\{
  \begin{array}{ll}
    (i-j+k,l,F), & \hbox{якщо~} j\leqslant k; \\
    (i,j-k+l,F), & \hbox{якщо~} j\geqslant k,
  \end{array}
\right.
\end{align*}
а тоді з означення напівгрупової операції на біциклічному моноїді ${\mathscr{C}}(p,q)$ випливає, що відображення $\mathfrak{f}\colon \boldsymbol{B}_{\omega}^{\mathscr{F}}\to {\mathscr{C}}(p,q)$, означене за формулою $\mathfrak{f}(i,j,F)=q^ip^j$ є ізоморфізмом.

\smallskip
$(\Longrightarrow)$
Припустимо, що напівгрупа $\boldsymbol{B}_{\omega}^{\mathscr{F}}$ ізоморфна біциклічному моноїду ${\mathscr{C}}(p,q)$. Оскільки біциклічний моноїд не містить нуля, то $\varnothing\not\in \mathscr{F}$. Також з теореми \ref{theorem-3.10} випливає, що сім'я $\mathscr{F}$ містить індуктивну множину $F_0$ в $\omega$ таку, що $F_0=\bigcup\mathscr{F}$. Розглянемо довільну множину $F\in \mathscr{F}$.  Оскільки біциклічний моноїд є біпростою інверсною напівгрупою та напівгрупа $\boldsymbol{B}_{\omega}^{\mathscr{F}}$ ізоморфна біциклічному моноїду ${\mathscr{C}}(p,q)$, то $F=F_0$ за теоремою  \ref{theorem-3.7}, а отже, $\mathscr{F}=\{F_0\}$.
\end{proof}

З твердження~\ref{proposition-3.11} випливає наслідок \ref{corollary-1114}.

\begin{corollary}\label{corollary-1114}
Нехай $\mathscr{F}$ --- ${\omega}$-замкнена підсім'я в  $\mathscr{P}(\omega)$. Якщо сім'я $\mathscr{F}$ містить непорожню індуктивну підмножину в $\omega$, то напівгрупа $\boldsymbol{B}_{\omega}^{\mathscr{F}}$ містить ізоморфну копію біциклічного моноїда.
\end{corollary}

\begin{theorem}\label{theorem-3.12}
Нехай $\mathscr{F}$ --- ${\omega}$-замкнена підсім'я в  $\mathscr{P}(\omega)$. Тоді такі умови екві\-ва\-лентні:
\begin{itemize}
  \item[$\boldsymbol{(i)}$] $\boldsymbol{B}_{\omega}^{\mathscr{F}}$~--- біпроста напівгрупа;
  \item[$\boldsymbol{(ii)}$] $\mathscr{F}$ --- одноелементна сім'я;
  \item[$\boldsymbol{(iii)}$] напівгрупа $\boldsymbol{B}_{\omega}^{\mathscr{F}}$ або тривіальна, або ізоморфна біциклічному моноїду.
\end{itemize}
\end{theorem}

\begin{proof}
Еквівалентність умов $\boldsymbol{(i)}$ і $\boldsymbol{(ii)}$ випливає з теореми \ref{theorem-3.7}$\boldsymbol{(iv)}$.

\smallskip
Імплікація $\boldsymbol{(iii)}\Longrightarrow\boldsymbol{(ii)}$ випливає з твердження \ref{proposition-3.11}.

\smallskip
Доведемо імплікацію $\boldsymbol{(ii)}\Longrightarrow\boldsymbol{(iii)}$. Якщо $\mathscr{F}=\{\varnothing\}$, то з означення напівгрупи $\boldsymbol{B}_{\omega}^{\mathscr{F}}$ випливає, що $\boldsymbol{B}_{\omega}^{\mathscr{F}}$ --- тривіальна (одноелементна) напівгрупа. Тому припустимо, що $\mathscr{F}=\{F\}$ для деякої непорожньої множини $F$. Позаяк сім'я $\mathscr{F}$ --- ${\omega}$-замкнена, то
$(-1+F)\cap F=F$, а отже, за лемою \ref{ind-sets}, $F$~--- індуктивна підмножина в $\omega$. Далі скористаємося твердженням \ref{proposition-3.11}.
\end{proof}

Якщо $\lambda$~--- ненульовий кардинал, то множина $\boldsymbol{\mathcal{B}}_\lambda=(\lambda\times\lambda)\sqcup\{0\}$ з напівгруповою операцією
\begin{equation*}
  (a,b)\cdot(c,d)=
\left\{
  \begin{array}{cl}
    (a,d), & \hbox{якщо~} b=c; \\
    0,     & \hbox{якщо~} b\neq c,
  \end{array}
\right.
\qquad \hbox{i} \qquad (a,b)\cdot 0=0\cdot(a,b)=0\cdot 0=0,
\end{equation*}
де $a,b,c,d\in\lambda$, називається \emph{напівгрупою $\lambda{\times}\lambda$-матричних одиниць} \cite{Lawson-1998, Petrich1984}.

\begin{proposition}\label{proposition-3.13}
Нехай $\mathscr{F}$ --- ${\omega}$-замкнена підсім'я в  $\mathscr{P}(\omega)$.
Напівгрупа $\boldsymbol{B}_{\omega}^{\mathscr{F}}$ ізо\-морфна напівгрупі $\omega{\times}\omega$-матричних одиниць $\boldsymbol{\mathcal{B}}_\omega$ тоді і тільки тоді, коли $\mathscr{F}=\{F,\varnothing\}$, де $F$~--- одноточкова підмножина в $\omega$.
\end{proposition}

\begin{proof}
$(\Longleftarrow)$
Припустимо, що $\mathscr{F}=\{F,\varnothing\}$ i $F$~--- одноточкова підмножина в $\omega$. Для довільних $i,j,k,l\in\omega$ отримаємо
\begin{align*}
  (i,j,F)\cdot (k,l,F)&=
\left\{
  \begin{array}{ll}
    (i-j+k,l,(j-k+F)\cap F), & \hbox{якщо~} j<k; \\
    (i,l,F\cap F),            & \hbox{якщо~} j=k\\
    (i,j-k+l,F\cap (k-j+F)), & \hbox{якщо~} j>k
  \end{array}
\right.= \\
  &=
\left\{
  \begin{array}{cl}
    (i,l,F),  & \hbox{якщо~} j=k \\
    0,        & \hbox{якщо~} j\neq k,
  \end{array}
\right.
\end{align*}
і очевидно, що
\begin{equation*}
(i,j,F)\cdot 0=0\cdot(i,j,F)=0\cdot 0=0,
\end{equation*}
звідки випливає, що відображення $\mathfrak{f}\colon \boldsymbol{B}_{\omega}^{\mathscr{F}}\to \boldsymbol{\mathcal{B}}_\omega$, означене за формулою $\mathfrak{f}(i,j,F)=(i,j)$ є ізоморфізмом.

\smallskip
$(\Longrightarrow)$
Припустимо, що напівгрупа $\boldsymbol{B}_{\omega}^{\mathscr{F}}$ ізоморфна напівгрупі $\omega{\times}\omega$-матричних одиниць $\boldsymbol{\mathcal{B}}_\omega$. Зафіксуємо довільний ненульовий елемент $(i,j,F)$ напівгрупи $\boldsymbol{B}_{\omega}^{\mathscr{F}}$. Якщо $(i,j,F)$ не є ідемпотентом, то за лемою~\ref{lemma-3.2} маємо, що $i\neq j$, а тоді
\begin{equation*}
  0=(i,j,F)\cdot(i,j,F)=
\left\{
  \begin{array}{ll}
    (i-j+i,k,(j-i+F)\cap F), & \hbox{якщо~} j<k; \\
    (i,j-i+j,F\cap (i-j+F)), & \hbox{якщо~} j>k.
  \end{array}
\right.
\end{equation*}
Отже, для довільних різних $i,j\in\omega$ одержимо, що
\begin{equation*}
(j-i+F)\cap F=\varnothing=F\cap (i-j+F),
\end{equation*}
а це означає, що для довільного натурального числа $k$ маємо, що $-k+F\cap F=\varnothing$. Звідси випливає, що множина $F$ одноелементна.

Припустимо, що сім'я $\mathscr{F}$ містить дві одноелементні множини $F_k=\{k\}$ i $F_l=\{l\}$, для деяких різних $k,l\in \omega$. Не зменшуючи загальності, можемо вважати, що $k<l$. За лемою~\ref{lemma-3.2} для довільного $i\in\omega$ елементи $(i+k,i+k,F_k)$ i $(i+l,i+l,F_l)$ є ідемпотентами напівгрупи $\boldsymbol{B}_{\omega}^{\mathscr{F}}$, причому $(i+k,i+k,F_k)\neq(i+l,i+l,F_l)$, оскільки $k\neq l$. Позаяк напівгрупи $\boldsymbol{B}_{\omega}^{\mathscr{F}}$ і  $\boldsymbol{\mathcal{B}}_\omega$ ізоморфні, то всі ненульові ідемпотенти в $\boldsymbol{B}_{\omega}^{\mathscr{F}}$ примітивні, а отже, отримуємо, що
\begin{align*}
  0&=(i+k,i+k,F_k)\cdot(i+l,i+l,F_l)=\\
   &=(i+l,i+l,(k-l+F_l)\cap F_k)=\\
   &=(i+l,i+l,F_k)\neq 0,
\end{align*}
суперечність. З отриманої суперечності випливає, що сім'я $\mathscr{F}$ містить лише одну одноелементну множину.
\end{proof}

Для довільного натурального числа $n$ позначимо $n\omega=\{n\cdot i\colon i\in\omega\}$.

\begin{lemma}\label{lemma-shift}
Нескінченна підмножина $F\subseteq \omega$ задовольняє умову
\begin{itemize}
  \item[($*$)] $(-k+F)\cap F=F$ або $(-k+F)\cap F=\varnothing$ для довільного натурального числа $k$
\end{itemize}
тоді і лише тоді, коли $F=i_0+n\omega$ для деяких натурального числа $n$ та $i_0\in\omega$.
\end{lemma}

\begin{proof}
Імплікація $(\Longleftarrow)$ очевидна.

\smallskip
$(\Longrightarrow)$ Припустимо, що нескінченна підмножина $F\subseteq \omega$ задовольняє умову ($*$). Очевидно, оскільки $F$~--- нескінченна множина, то існує найменше натуральне число $n_0$ таке, що $(-n_0+F)\cap F\neq\varnothing$, а отже,  $(-n_0+F)\cap F=F$. Якщо $n_0=1$, то $F$~--- індуктивна підмножина в $\omega$ за лемою \ref{ind-sets}, а отже, $F=i_0+\omega$, де $i_0=\min F$.

Далі вважатимемо, що $n_0>1$. Якщо $j_0=\min F$, то $j_0+n_0\in F$, а отже $j_0+n_0\omega\subseteq F$. Припустимо, що існує число $m\in\omega\setminus(j_0+n_0\omega)$ таке, що $m\in F$, і нехай $j_0+kn_0<m<j_0+(k+1)n_0$ для деякого натурального числа $k$. Тоді
\begin{equation*}
m-(j_0+n_0),\ldots,m-(j_0+kn_0)\in F.
\end{equation*}
Однак $m-(j_0+kn_0)<n_0$, а отже існує таке натуральне число $m-(j_0+kn_0)$, що
\begin{equation*}
-(m-(j_0+kn_0))+F\subseteq F,
\end{equation*}
що суперечить вибору числа $n_0$. З отриманої суперечності випливає  імплікація ${(\Longrightarrow)}$.
\end{proof}

Нагадаємо \cite{Lawson-1998}, шо інверсна напівгрупа $S$ з нулем $0$ називається \emph{$0$-біпростою}, якщо $S$ має лише два $\mathscr{D}$-класи: $S\setminus\{0\}$ i $\{0\}$.

\begin{example}\label{example-3.15}
Зафіксуємо довільні $i_0\in\omega$ та натуральне число $j_0$. Приймемо $\boldsymbol{B}_{\omega}^{(i_0,j_0)}=\boldsymbol{B}_{\omega}^{\mathscr{F}}$, де $\mathscr{F}=\left\{\varnothing,i_0+j_0\omega\right\}$. Тоді, очевидно, що $\mathscr{F}$~--- $\omega$~-замкнена сім'я в  $\mathscr{P}(\omega)$, а також за теоремою~\ref{theorem-3.5} інверсна напівгрупа $\boldsymbol{B}_{\omega}^{(i_0,j_0)}$ є $0$-біпростою. Більше того, для довільного $i_0\in\omega$ за твердженням \ref{proposition-3.11} напівгрупа $\boldsymbol{B}_{\omega}^{(i_0,1)}$ ізоморфна біциклічному моноїду з приєднаним нулем.
\end{example}

Безпосередньо звичайною перевіркою доводиться, що для довільних $i_1,i_2\in\omega$ та довільного натурального числа $j_0$ відображення $\mathfrak{h}\colon \boldsymbol{B}_{\omega}^{(i_1,j_0)}\to \boldsymbol{B}_{\omega}^{(i_2,j_0)}$, означене
\begin{equation*}
  \mathfrak{h}(n,m,i_1+j_0\omega)=(n,m,i_2+j_0\omega) \qquad \hbox{i} \qquad \mathfrak{h}(0)=0
\end{equation*}
є ізоморфізмом. Отже, виконується твердження \ref{proposition-3.16}.

\begin{proposition}\label{proposition-3.16}
Для довільних $i_1,i_2\in\omega$ та довільного натурального числа $j_0$ напівгрупи $\boldsymbol{B}_{\omega}^{(i_1,j_0)}$ i $\boldsymbol{B}_{\omega}^{(i_2,j_0)}$ ізоморфні.
\end{proposition}

Теорема \ref{theorem-3.17} описує структуру $0$-біпростих інверсних напівгруп $\boldsymbol{B}_{\omega}^{\mathscr{F}}$ з точністю до ізоморфізму напівгруп.

\begin{theorem}\label{theorem-3.17}
Нехай $\mathscr{F}$ --- ${\omega}$-замкнена підсім'я в  $\mathscr{P}(\omega)$, $\varnothing\in\mathscr{F}$ i $\boldsymbol{B}_{\omega}^{\mathscr{F}}$~--- $0$-біпроста напівгрупа. Тоді виконується лише одна з умов:
\begin{itemize}
  \item[$(1)$] напівгрупа $\boldsymbol{B}_{\omega}^{\mathscr{F}}$ ізоморфна напівгрупі $\omega{\times}\omega$-матричних одиниць $\boldsymbol{\mathcal{B}}_\omega$;
  \item[$(2)$] напівгрупа $\boldsymbol{B}_{\omega}^{\mathscr{F}}$ ізоморфна біциклічному моноїду з приєднаним нулем;
  \item[$(3)$] напівгрупа $\boldsymbol{B}_{\omega}^{\mathscr{F}}$ ізоморфна напівгрупі $\boldsymbol{B}_{\omega}^{(0,j_0)}$ для деякого натурального числа $j_0\geqslant 2$.
\end{itemize}
\end{theorem}

\begin{proof}
За теоремою \ref{theorem-3.7}$\boldsymbol{(iv)}$ сім'я $\mathscr{F}$ містить непорожню множину $F$ і порож\-ню множину $\varnothing$. Припустимо, що множина $F$ скінченна. Тоді, аналогічно, як і в твердженні \ref{proposition-3.13} доводиться, що $F$ одноелементна множина, а отже, за твердженням \ref{proposition-3.13} напівгрупа $\boldsymbol{B}_{\omega}^{\mathscr{F}}$ ізоморфна напівгрупі $\omega{\times}\omega$-матричних одиниць $\boldsymbol{\mathcal{B}}_\omega$.

Якщо ж $F$ --- нескінченна множина, то з теореми \ref{theorem-3.7}$\boldsymbol{(iv)}$ випливає, що викону\-ється одна з умов
\begin{equation*}
a)~(-1+F)\cap F=F \qquad \hbox{або} \qquad b)~(-1+F)\cap F=\varnothing.
\end{equation*}

У випадку $a)$ за твердженням \ref{proposition-3.11} множина $\boldsymbol{B}_{\omega}^{\mathscr{F}}\setminus\{0\}$ є піднапівгрупою в $\boldsymbol{B}_{\omega}^{\mathscr{F}}$, яка ізоморфна біциклічному моноїду, а отже, виконується твердження $(2)$.

Якщо ж виконується умова $b)$, то за лемою \ref{lemma-shift} матимемо, що $F=j_0+i_0\omega$ для деяких натурального числа $i_0$ та $j_0\in\omega$. Застосувавши твердження \ref{proposition-3.16}, отримуємо, що виконується твердження $(3)$.
\end{proof}

Нагадаємо \cite{Lawson-1998, Petrich1984}, що \emph{найменша} (\emph{мінімальна}) \emph{групова конгруенція} $\boldsymbol{\sigma}$ на ін\-версній напівгрупі $S$ визначається так:
\begin{equation*}
  s\boldsymbol{\sigma}t \qquad \Longleftrightarrow \qquad es=et \quad \hbox{для деякого} \quad e\in E(S).
\end{equation*}

Очевидно, якщо $\mathscr{F}$ --- ${\omega}$-замкнена підсім'я в  $\mathscr{P}(\omega)$ i $\varnothing\in\mathscr{F}$, то напівгрупа $\boldsymbol{B}_{\omega}^{\mathscr{F}}$ містить нуль, а отже, фактор-напівгрупа $\boldsymbol{B}_{\omega}^{\mathscr{F}}/\boldsymbol{\sigma}$ ізоморфна тривіальній напівгрупі.

\begin{proposition}\label{proposition-3.19}
Нехай $\mathscr{F}$ --- ${\omega}$-замкнена підсім'я в  $\mathscr{P}(\omega)$ i $\varnothing\notin\mathscr{F}$. Тоді \linebreak $(i_1,j_1,F_1)\boldsymbol{\sigma}(i_2,j_2,F_2)$ в $\boldsymbol{B}_{\omega}^{\mathscr{F}}$ тоді і тільки тоді, коли $i_1-j_1=i_2-j_2$, а отже, фактор-напівгрупа $\boldsymbol{B}_{\omega}^{\mathscr{F}}/\boldsymbol{\sigma}$ ізоморфна адитивній групі цілих чисел $\mathbb{Z}(+)$.
\end{proposition}

\begin{proof}
Нехай $(i_1,j_1,F_1)$ i $(i_2,j_2,F_2)$ --- довільні  елементи напівгрупи $\boldsymbol{B}_{\omega}^{\mathscr{F}}$. З означення найменшої групової конгруенції $\boldsymbol{\sigma}$ випливає, що $(i_1,j_1,F_1)\boldsymbol{\sigma}(i_2,j_2,F_2)$ тоді і тільки тоді, коли існує елемент $(i,j,F)\in\boldsymbol{B}_{\omega}^{\mathscr{F}}$ такий, що $(i,j,F)\preccurlyeq(i_1,j_1,F_1)$ i $(i,j,F)\preccurlyeq(i_2,j_2,F_2)$. Тоді за твердженням~\ref{proposition-3.8} отримаємо, що
\begin{equation*}
F\subseteq -k_1+F_1 \qquad \hbox{i} \qquad i-i_1=j-j_1=k_1
\end{equation*}
та
\begin{equation*}
F\subseteq -k_2+F_2 \qquad \hbox{i} \qquad i-i_2=j-j_2=k_2
\end{equation*}
для деяких $k_1,k_2\in\omega$. Отож з $(i_1,j_1,F_1)\boldsymbol{\sigma}(i_2,j_2,F_2)$ в $\boldsymbol{B}_{\omega}^{\mathscr{F}}$ випливає, що
\begin{equation*}
i_1-j_1=i_2-j_2=i-j.
\end{equation*}

Припустимо, що для елементів $(i_1,j_1,F_1)$ i $(i_2,j_2,F_2)$  напівгрупи $\boldsymbol{B}_{\omega}^{\mathscr{F}}$ справджується рівність $i_1-j_1=i_2-j_2$. Не зменшуючи загальності можемо вважати, що $i_1>i_2$. Позаяк $\mathscr{F}$ --- ${\omega}$-замкнена підсім'я в  $\mathscr{P}(\omega)$, то
\begin{equation*}
F=F_1\cap (i_2-i_1+F_2)\in \mathscr{F}.
\end{equation*}
Тоді $j_1>j_2$ і за твердженням~\ref{proposition-3.8} отримаємо, що
\begin{align*}
  (i_1,j_1,F)&=(i_1,j_1,F\cap F_1)=(i_1,i_1,F)\cdot(i_1,j_1,F_1) \preccurlyeq(i_1,j_1,F_1)
\end{align*}
i
\begin{align*}
    (i_1,i_1,F)\cdot (i_2,j_2,F_2)&=(i_1,i_1-i_2+j_2,F\cap(i_2-i_1+F_2))=\\
                                  &=(i_1,j_1-j_2+j_2,F\cap(i_2-i_1+F_2))=\\
                                  &=(i_1,j_1,F\cap(i_2-i_1+F_2))=\\
                                  &=(i_1,j_1,F)\preccurlyeq\\
                                  &\preccurlyeq(i_2,j_2,F_2),
\end{align*}
а отже, $(i_1,j_1,F_1)\boldsymbol{\sigma}(i_2,j_2,F_2)$ в $\boldsymbol{B}_{\omega}^{\mathscr{F}}$.

Означимо відображення $\mathfrak{h}_{\boldsymbol{\sigma}}\colon \boldsymbol{B}_{\omega}^{\mathscr{F}} \to Z(+)$ за формулою $\mathfrak{h}_{\boldsymbol{\sigma}}(i,j,F)=i-j$. З вище доведеного випливає, що $\mathfrak{h}_{\boldsymbol{\sigma}}(i_1,j_1,F_1)=\mathfrak{h}_{\boldsymbol{\sigma}}(i_2,j_2,F_2)$ тоді і лише тоді, коли $(i_1,j_1,F_1)\boldsymbol{\sigma}(i_2,j_2,F_2)$ в $\boldsymbol{B}_{\omega}^{\mathscr{F}}$, а отже, відображення $\mathfrak{h}_{\boldsymbol{\sigma}}\colon \boldsymbol{B}_{\omega}^{\mathscr{F}} \to Z(+)$ є гомоморфізмом і фактор-напівгрупа $\boldsymbol{B}_{\omega}^{\mathscr{F}}/\boldsymbol{\sigma}$ ізоморфна адитивній групі цілих чисел $\mathbb{Z}(+)$.
\end{proof}

%\bigskip

\section*{\textbf{Подяка}}

Автори висловлюють щиру подяку  рецензентов за цінні поради.
%%%%%%%%%%%%%%%%%%%%%%%%%%%%%%%%%%%%%%%%%%%%%%%%%%%%%%%%%%%%%%%%%%%%%%%%%%%%%%%%%%

%\vskip1cm
\end{document}